\documentclass[12pt,reqno]{amsart}
\usepackage{amsmath,amsfonts,amssymb,calligra,amscd,amsthm,amsbsy,epsf}
\usepackage[dvips]{graphicx}
\usepackage{comment}
\usepackage[dvips]{color}

\numberwithin{equation}{section}
\newtheorem{theorem}{Theorem}
\newtheorem{meta-thm}[theorem]{Meta-Theorem}
\newtheorem{lemma}[theorem]{Lemma}
\newtheorem{cor}[theorem]{Corollary}

\newtheorem{remark}[theorem]{Remark}
\newtheorem{definition}[theorem]{Definition}
\newtheorem{question}[theorem]{Question}

\begin{document}
\title[Constant periodic data and entropy of Anosov diffeomorphisms]
{Constant periodic data and entropy of Anosov diffeomorphisms}
\author[F. Micena]{Fernando Pereira Micena}
\address{Instituto de Matem\'{a}tica e Computa\c{c}\~{a}o,
  IMC-UNIFEI, Itajub\'{a}-MG, Brazil.}
\email{fpmicena82@unifei.edu.br}


\baselineskip=18pt              


\begin{abstract}

We study the effects that the constant periodic data condition  have on topological entropy of Anosov diffeomorphisms. Under constant periodic data condition we prove that
Anosov diffeomorphism has finitely many measures of maximal entropy and each one of them is absolutely continuous with respect to Lebesgue. From this, in the setting of $C^{\infty}-$Anosov diffeomorphisms satisfying constant periodic data,  we provide a characterization of transitivity property via measures of maximal entropy.
\end{abstract}

\subjclass[2010]{}

\maketitle


\section{Introduction}\label{sec:intro}

Let $M$ be a $C^{\infty}$ compact, connected and boundaryless manifold and $f:M \rightarrow M$ be a diffeomorphism. We say that $f$ is an Anosov diffeomorphims if there are numbers $0 < \beta < 1 < \eta, C > 0$ and a $Df-$invariant continuous splitting $T_xM = E^u_f(x) \oplus E^s_f(x),$ such that
$$ ||Df^n(x) \cdot v || \geq \frac{1}{C} \eta^n ||v||, \forall v \in E^u_f(x),$$
$$ ||Df^n(x) \cdot v || \leq C \beta^n ||v||, \forall v \in E^s_f(x).$$

Anosov diffeomorphims play an important role in the theory of dynamical systems and this class of diffeomorphisms satisfies many rich dynamical properties, as shadowing and closing lemmas and, in the case, that the Anosov diffeomorphism is $C^2$ and preserves a mesaure $\mu $ absolutely continuous, this measure is ergodic. Frequently when we study Anosov diffeomorphisms we assume transitivity.

\begin{definition}
Let $(X,d)$ be a compact metric space and $f: X \rightarrow X$ a continuous function. We say that $f$ is transitive if for any nonempty open sets $U$ and $V$ there exists an integer $N$ such that $f^{-N}(V) \cap U \neq \emptyset,$ or equivalently, there exists a point $x \in X$ with dense orbit.
\end{definition}

For Anosov diffeomorphisms, transitivity property is equivalent to $\Omega(f) = M,$ where $\Omega(f)$ is the non wandering set of $f.$
Under transitivity assumption, Anosov diffeomorphisms have a unique measure of maximal entropy, see \cite{Bowen}, for instance. We could ask about the reciprocal.

\begin{question}\label{quest}
Given $f: M \rightarrow M$ an Anosov diffeomorphisms with unique measure of maximal entropy, is $f$ transitive ?
\end{question}

In general this question is totally inconclusive and it is related with a more general question about the transitivity of Anosov diffeomorphisms. It is known by \cite{FW} examples of Anosov flows which are not transitive, however it is not known if all Anosov diffeomorphisms are transitive.  We could ask under what conditions a $C^2-$Anosov diffeomorphism with a unique measure of maximal entropy is transitive. Note that, in case  of a non transitive Anosov diffeomorphism,   could occur  $\Omega(f) = \Omega_1 \cup \ldots \cup \Omega_s, s > 1,$ a union of distinct basic sets with different topological entropies, so $f$ could have a unique measure of maximal entropy but would not be transitive.

Under some regular conditions, we are able to give a positive answer for the above question. Let us to point the constant periodic data (c.p.d) condition.

\begin{definition}\label{cpd}
Let $f: M \rightarrow M$ be a local diffeomorphism. We say that $f$
has constant periodic data if for any periodic points $p,q$ of $f,$
with period $k$ and $n$ respectively, then $Df^{\tau}(p)$ and
$Df^{\tau}(q)$ are conjugated matrixes, for every integer $\tau$ such that
$f^{\tau}(p) = p$ and $f^{\tau}(q) = q.$ In particular the set of
Lyapunov exponents of $p$ and $q,$ are equal and each common
Lyapunov exponent has the same multiplicity for both.
\end{definition}

Given $f: M \rightarrow M,$ a transitive $C^{\infty}-$Anosov diffeomorphism satisfying constant periodic data condition, by Lemma \ref{lem3-rafael}, $f$ is such that $Jf^n(p) = 1,$ for any $p$ such that $f^n(p) = p.$ In particular $f$ preserves a measure $\mu$ absolutely continuous such that $h_{\mu}(f) = \Lambda^u_f$ (the constant sum of Lyapunov exponents) and consequently $\mu$ is the unique measure of maximal entropy of $f.$ From Livsic's theorem and \cite{LlaveMM86}, this measure has $C^{\infty}-$density (Lemma \ref{lemmafinal}). In our case, we prove the converse asked in Question \ref{quest}, requiring minimal regularity on the density.

Under the constant periodic data condition we can prove the following results.

\begin{theorem}\label{t1}
Let $f: M \rightarrow M$ be a $C^\infty-$Anosov diffeomorphism satisfying the constant periodic data condition. Then $f$ has finitely many measures of maximal entropy, such that each one of them is absolutely continuous with respect to Lebesgue measure. Moreover, if $f$ has
a unique measure of maximal entropy, with continuous density with respect to Lebesgue
measure, then this measure is equivalent to Lebesgue and the density is in fact $C^{\infty}.$
\end{theorem}

\begin{theorem}\label{t2}
Let $f: M \rightarrow M$ be a $C^\infty-$Anosov diffeomorphism such that every point is regular for $f.$ Then $f$ has finitely many measures of maximal entropy, such that each one of them is absolutely continuous with respect to Lebesgue measure. Moreover, if $f$ has
a unique measure of maximal entropy, with continuous density with respect to Lebesgue
measure, then this measure is equivalent to Lebesgue and the density is in fact $C^{\infty}.$
\end{theorem}

Let us to describe how Theorem \ref{t2} follows from Theorem \ref{t1}. For each $x \in M,$ denote by $\Lambda^u_f(x)$ the sum o positive Lyapunov exponents of $f$ at $x.$

Let $x_0$ be an arbitrary point on $M$ and consider $ \Lambda^u_f = \Lambda^u_f(x_0).$ Since $f$ have local product structure, there is an open neighborhood $V_{x_0}$ of $x_0,$ such that,  given $z \in V,$ there is a point $z' \in W^u_f(z)\cap W^s_f(x_0),$ so since every point is regular we have

$$ \Lambda^u_f(x_0) = \displaystyle\lim_{n \rightarrow +\infty}\frac{1}{n} \log(|J^uf^n(x_0)|) =
\displaystyle\lim_{n \rightarrow +\infty}\frac{1}{n} \log(|J^uf^n(z')|) = \displaystyle\lim_{n \rightarrow -\infty}\frac{1}{n} \log(|J^uf^n(z)|) =  \Lambda^u_f(z), $$
where $J^uf^n(x) = |\det(Df^n(x): E^u_f(x) \rightarrow E^u_f(f^n(x)))|.$

The map $x \mapsto \Lambda^u_f(x)$ is locally constant, since $M$ is connect it implies  $\Lambda^u_f(x) = \Lambda^u_f,$ for any $x \in M.$ Particularly $f$ satisfies the constant periodic data condition.

The above Theorems are generalization of some results in \cite{MT}. For more results about
the number of SRB measures of local diffeomorphisms and how it is related to basic sets,
see \cite{LM20}.

Direct from above Theorems we can get the following corollary.

\begin{cor}\label{c1}
Let $f: M \rightarrow M$ be a $C^\infty-$Anosov diffeomorphism satisfying constant periodic data condition  (or every point is regular for $f$).  Then $f$ is transitive if and only if $f$ has a unique measure of maximal entropy, absolutely continuous with respect to Lebesgue measure with continuous density.
\end{cor}

The above results can be formulated in $C^{1 + \alpha}, \alpha > 0,$ context. It is did at final, see Theorem \ref{alfa}.

In $C^1$ setting we are able to prove the following theorem.

\begin{theorem}\label{t3}
Let $f: M \rightarrow M$ be a $C^1$ Anosov diffeomorphism  satisfying constant periodic data condition (or such that every point is regular for $f$). If $f$ has
a unique measure of maximal entropy, which is absolutely continuous with respect Lebesgue measure with continuous density,
then $f$ is transitive.
\end{theorem}

\section{Useful Tools}

To prove the  Theorem \ref{t1} we will need some tools
involving regularity of all point and uniform convergence of Lyapunov
exponents. Uniform convergence of Lyapunov exponents is related
with volume growth unstable foliation and consequently with the entropy along
the unstable leaves.

Let us recall results  given in \cite{AAS} and \cite{Cao}.

\begin{lemma} \label{lemmauniform1} Let $\mathcal{M}$ be the space of $f-$invariant measures, $\phi$ be a continuous function on $M.$ If $\int \phi d\mu < \lambda, \; \forall \mu \in \mathcal{M}, $ then for every $x \in M,$ there exists $n(x)$ such that
$$\frac{1}{n(x)} \sum_{i = 0}^{n(x) - 1} \phi(f^i(x)) < \lambda. $$
\end{lemma}

\begin{lemma}\label{lemmauniform2} Let $\mathcal{M}$ be the space of $f-$invariant measures, $\phi$ be a continuous function on $M.$ If $\int \phi d\mu < \lambda, \; \forall \mu \in \mathcal{M}, $ then  there exists $N$ such that for all $n \geq N,$ we have
$$\frac{1}{n} \sum_{i = 0}^{n - 1} \phi(f^i(x)) < \lambda,$$
 for all $x \in M.$
\end{lemma}

See \cite{Cao} for the proofs of the above Lemmas. In the previous lemmas if we replace  $\int \phi d\mu < \lambda$ by $\int \phi d\mu > \lambda,$ we can get analogous statements.

\begin{lemma} \label{lemmauniform3} Let $\mathcal{M}$ be the space of $f-$invariant measures, $\phi$ be a continuous function on $M.$ If $\int \phi d\mu > \lambda, \; \forall \mu \in \mathcal{M}, $ then for every $x \in M,$ there exists $n(x)$ such that
$$\frac{1}{n(x)} \sum_{i = 0}^{n(x) - 1} \phi(f^i(x)) > \lambda. $$
\end{lemma}

\begin{lemma}\label{lemmauniform4} Let $\mathcal{M}$ be the space of $f-$invariant measures, $\phi$ be a continuous function on $M.$ If $\int \phi d\mu > \lambda, \; \forall \mu \in \mathcal{M}, $ then  there exists $N$ such that for all $n \geq N,$ we have
$$\frac{1}{n} \sum_{i = 0}^{n - 1} \phi(f^i(x)) > \lambda,$$
 for all $x \in M.$
\end{lemma}

In \cite{Hua} the authors lead with a notion of topological entropy $h_{top}(f, \mathcal{W} )$ of an invariant expanding foliation $\mathcal{W}$ of a diffeomorphism $f. $ They establish variational principle in this sense and relation between $h_{top}(f, \mathcal{W} )$ and volume growth of $\mathcal{W}. $

Here $W(x)$ will denote the leaf of $\mathcal{W}$ by $x.$ Given  $\delta  > 0,$  we denote by $W(x, \delta)$ the $\delta-$ball centered in $x$ on $W(x),$ with the induced riemannian distance, that we will denote by $d_{W}.$

Given $x \in M, $ $\varepsilon > 0, $ $\delta > 0$ and $n \geq 1$ a integer number, let $N_{W}(f, \varepsilon, n, x, \delta)$ be the maximal cardinality of a set $S \subset \overline{W(x, \delta)}$ such that $\displaystyle\max_{j =0 \ldots, n-1} d_{W}(f^j(z), f^j(y)) \geq \varepsilon,$ for any distinct points $y,z \in S.$

\begin{definition}\label{uentropy} The unstable entropy of $f$ on $M,$ with respect to the expanding foliation $\mathcal{W}$ is given by
$$h_{top}(f, \mathcal{W} ) = \lim_{\delta \rightarrow 0} \sup_{x \in M} h^{\mathcal{W}}_{top}(f, \overline{W(x, \delta)}), $$
where
$$h^{\mathcal{W}}_{top}(f, \overline{W(x, \delta)}) = \lim_{\varepsilon \rightarrow 0} \limsup_{n \rightarrow +\infty} \frac{1}{n} \log(N_{W}(f, \varepsilon, n, x, \delta)). $$
\end{definition}

Define $\mathcal{W}-$volume growth by
 $$\chi_{\mathcal{W}}(f) = \sup_{x \in M } \chi_{\mathcal{W}}(x, \delta), $$
where
$$ \chi_{\mathcal{W}}(x, \delta) = \limsup_{n\rightarrow +\infty} \frac{1}{n} \log(Vol(f^n(W(x, \delta)))).$$

Note that, since we are supposing $\mathcal{W}$ a expanding foliation, the above definition is independent of $\delta$ and the riemannian metric.

\begin{theorem}[Theorem C and Corollary C.1 of \cite{Hua}]\label{teoH} With the notations above we have
$$h_{top}(f, \mathcal{W} ) = \chi_{\mathcal{W}}(f).$$
Moreover $h_{top}(f) \geq h_{top}(f, \mathcal{W}). $
\end{theorem}

\section{Proof of Theorem \ref{t1}}

We provide the proof of Theorem \ref{t1} following some lemmas.

%
%
%
%
%
%
%
%

\begin{lemma}\label{lem3-rafael}
Let $f$ be a $C^\infty$ Anosov diffeomorphism of
a compact manifold $M.$ Assume that $f$ satisfies the constant periodic data condition. Then $Jf^n(p) = 1,$ for any $p$ such that $f(p) = p,$ moreover  $f$ has finitely many measures
of maximal entropy and each one of them is absolutely continuous with respect to Lebesgue.
\end{lemma}

\begin{proof}

If $p$ is a periodic point of $f,$ with period $n \geq 1,$ define $\Lambda^u_f(p) = \frac{1}{n} \log(J^uf^n(p)).$   Since we have c.p.d condition, there exists a number $\Lambda^u_f$ such that $\Lambda^u_f(p) = \Lambda^u_f,$ for any $p \in Per(f).$
Let $\mu$ be an $f-$invariant probability measure, and denote by $R$ the set of regular and recurrent points of $f,$ we have $\mu(R)  = 1.$ Since $\mu $ is a hyperbolic measure, by using Katok/Anosov Closing Lemma we get

\begin{equation}
\Lambda^u_f(x):= \lim_{n \rightarrow + \infty} \frac{1}{n} \log(J^uf^n(x)) = \Lambda^u_f, \label{sameu}
\end{equation}
for any $x \in R,$ hence the sum of positive Lyapunov exponents $\Lambda^u_f(x) \leq \Lambda^u_f. $ By Ruelle's inequality we have $h_{\mu}(f) \leq \Lambda^u_f,$ and finally by variational principle

\begin{equation}\label{topleq}
h_{top}(f) \leq \Lambda^u_f.
\end{equation}

Now using the Lemmas \ref{lemmauniform1} and \ref{lemmauniform2} with $\phi(x) = \log(J^uf(x)), x \in M,$ we conclude that the limit given in the expression \eqref{sameu} is uniform. So for any $\varepsilon > 0,$ there is $N > 0$ an integer number such that for any $n \geq N$ and $x \in M,$ we have
\begin{equation}
J^uf^n(x) > e^{n(\Lambda^u_f - \varepsilon)}.
\end{equation}
So given $B^u(x, \delta)$ a $u-$ball centered in $x$ with radius $\delta  >0,$ we have
\begin{equation}\label{growthu}
Vol_u(f^n(B^u(x, \delta))) = \int_{B^u(x, \delta)}  J^uf^n(x) dVol_u(x) > e^{n(\Lambda^u_f - \varepsilon)}Vol_u(B^u(x, \delta)),
\end{equation}
where $Vol_u$ denotes the $u-$dimension volume along unstable leaves induced by the riemannian metric of $M.$

By the equation \eqref{growthu} we ge $\chi_{\mathcal{W}^u}(f) \geq \Lambda^u - \varepsilon, $ for any $\varepsilon > 0.$ From Theorem \ref{teoH} we have $h_{top}(f, \mathcal{W}^u) \geq \Lambda^u_f,$ and
\begin{equation} \label{topgeq}
h_{top}(f)\geq \Lambda^u_f.
 \end{equation}

Now the equations \eqref{topleq} and \eqref{topgeq}, we get $h_{top}(f) = \Lambda^u_f. $  Analogously, taking the inverse $f^{-1},$ we conclude that $h_{top}(f) = -\Lambda^s_f.$ So $\Lambda^s_f + \Lambda^u_f = 0,$ it means that for any $p$ such that $f^n(p) = p,$  we have $Jf^n(p) = 1.$

Let $\nu$ be a maximal entropy measure of $f,$ we have

$$h_{\nu}(f) = \Lambda^u_f = \int_M \sum_{\lambda_i > 0} \lambda_i dim E_i d\nu, $$
$$ h_{\nu}(f) =  h_{\nu}(f^{-1})  = -\Lambda^s_f = -\int_M \sum_{\lambda_i < 0} \lambda_i dim E_i d\nu. $$

Using S.R.B theory, see \cite{LEDRAPPIER}, the above expressions means that $\nu$ has absolutely continuous density along the stable and unstable foliations, so $\nu$ is absolutely continuous. In fact, first we consider a borelian set  $C$  such that $m(C) = 0$ and $C$ intersects each $u-$leaf on a zero $u-$volume set, then $\nu(C) = 0.$ In the same way, for a boreian set $C,$ with $m(C)=0,$ such that $C$ intersects each $s-$leaf on a zero $s-$volume set, we have $\nu(C) = 0.$ Now, consider on a box $B$ a set $C \subset B$ such that $m(C) = 0.$ Consider $\tilde{C} = \{x \in C| \; Vol_u(C \cap W^u_x(B)) > 0\} ,$ where $W^u_x(B)$ is the connected component of $W^u_x \subset B$ that contains $x.$ Consider $\tilde{\tilde{C}} = \cup_{x \in \tilde{C}} W^u_x(B),$ the $u-$saturation of $\tilde{C}$ restricted to $B.$  Since the unstable foliation is transversely absolutely continuous we have $m(\tilde{\tilde{C}}) = 0$ and $\tilde{\tilde{C}}$ intersects each $W^s_x(B)$ on a zero $s-$volume set. The last affirmation is true because if for one arc $W^s_x(B)$ we could have $Leb_{W^s_x(B)}(\tilde{\tilde{C}}) > 0,$ by transverse absolute continuity of the unstable foliation we would have $Leb_{W^s_z(B)}(\tilde{\tilde{C}}) > 0,$ for any $z \in B,$ here $Leb_{W^s_z(B)}(A),$ denotes the Lebesgue meaure of $W^s_z(B).$ We conclude, $\nu(\tilde{C}) = 0$ and consequently $\nu(C) = 0.$ Thus measure $\nu$ is absolutely continuous w.r.t $m.$

Let us to show that there is a finite number of measures of maximal entropy. By spectral decomposition $$\Omega(f) = \Omega_1 \cup \ldots \cup \Omega_s$$ a finite union of basic sets such that $f_i := f|\Omega_i, i =1, \ldots, s$ is transitive. For each $f_i : \Omega_i \rightarrow \Omega_i,$ consider the unique measure of maximal entropy $\nu_i,$ given by Theorem 4.1 of \cite{Bowen}.
Now if $\nu$ is a measure of maximal entropy of $f,$ we have, by Poincar\'{e} Recurrence Theorem that $\nu (\Omega ) = 1,$ consequently $\nu(\Omega_i) > 0,$ for some $i =1, \ldots, s.$ Since $\nu$ is absolutely continous with respect to Lebesgue, then $\nu$ is ergodic and since $f(\Omega_i) = \Omega_i, \nu(\Omega_i) > 0,$ we conclude $\nu(\Omega_i) = 1.$
We get $$ h_{top}(f) = h_{\nu}(f) = h_{\nu}(f|\Omega_i) \leq h_{\nu_i}(f|\Omega_i)   \leq h_{top}(f),$$
so $$ h_{\nu_i}(f|\Omega_i) =  h_{\nu}(f) = h_{\nu}(f|\Omega_i) = h_{top}(f).$$
Since $f|\Omega_i$ has a unique measure of maximal entropy and $ \nu (\Omega(f)\setminus \Omega_i) = 0,$ we get $\nu = \nu_i.$ Each measure of maximal entropy coincide with a $\nu_i,$ for some $i =1, \ldots, s.$

\end{proof}

\begin{lemma}\label{fudamental}
Let $f$ be as in Theorem \ref{t1} and suppose that $\nu$ is the unique measure of maximal entropy with continuous density, then $\nu(U) > 0,$ for any open set $U \neq \emptyset,$ consequently $f$ is transitive.
\end{lemma}

\begin{proof}
Let $\rho$ be the continuous density of $\nu$ and consider the compact set $H_0 = \rho^{-1}(\{ 0\}).$ Suppose that there exists an open set $U \neq \emptyset,$ such that $\nu(U) = 0,$ so $U \subset H_0.$ In fact if it existed some $x \in U,$ such that $\rho(x) > 0,$ then by continuity of $\rho,$ we could choose an open set $V \subset U,$ such that $x \in V$ and $\rho(t) > \delta > 0,$ for some $\delta > 0$ and every $t \in V.$ With this $0< \nu(V) < \nu(U),$ that contradicts the fact $\nu(U) = 0.$

By invariance of $\nu,$ give $n \in \mathbb{Z},$ we have $f^n(U)$ is open set and $\nu(f^n(U)) = 0,$ so $f^n(U) \subset H_0.$ We conclude that $U \subset f^{-n}(H_0),$ for any $n \in \mathbb{Z}.$

Define $\Lambda = \displaystyle\bigcap_{n \in \mathbb{Z}} f^{-n}(H_0),$ we note that $U \subset \Lambda,$ and $\Lambda$ is a hyperbolic compact set for $f.$

Since $U \subset \Lambda,$ we can take a local unstable arc of $W^u_f(x) , x \in U.$ As in Lemma \ref{lem3-rafael}   we have
$ h_{top}(f|{\Lambda}, \mathcal{W}^u_{f|{\Lambda}}) \geq \Lambda^u_f ,$ and by Theorem \ref{teoH} we get

$$ h_{top}(f|{\Lambda}) \geq h_{top}(f|{\Lambda}, \mathcal{W}^u_{f|{\Lambda}}) \geq \Lambda^u_f = h_{top}(f) \geq h_{top}(f|{\Lambda}).$$

So $h_{top}(f|{\Lambda}) =  h_{top}(f),$ since $f|{\Lambda}$ is expansive, then $f|{\Lambda}$ admits a measure of maximal entropy $\nu,$ such that $\nu(\Lambda) = 1.$ Define $\overline{\mu}$ a borelian measure such that $\overline{\mu}(B) = \mu(B \cap \Lambda).$ We have $\overline{\mu}$ is a measure of maximal entropy of $f$ which is singular with respect to $\nu.$ It is an absurd, since $f$ has a unique measure of maximal entropy.

Since $f$ is $C^{\infty}$ and $\nu$ is absolutely continuous, the measure $\nu$ is ergodic. For any open set $U \neq \emptyset$ we have $\nu(U) > 0,$ given $V \neq \emptyset,$ another open set, by ergodicity for $\nu$ a.e $x \in U,$

$$ \displaystyle\lim_{n \rightarrow +\infty}\frac{1}{n} \displaystyle\sum_{j=1}^{n-1} \chi_{V} (f^j(x)) \rightarrow  \nu(V) > 0,$$
so $f$ is transitive.

\end{proof}

\begin{lemma}\label{lemmafinal}
If $\nu$ is the unique measure of maximal entropy of $f$ as in Theorem \ref{t1}, then $\nu$ has $C^{\infty}$ density with respect to Lebesgue measure.
\end{lemma}

\begin{proof}
We fix any smooth measure $\mu$ on the manifold equivalent to Lebesgue and
denote by $Jf(x)$ the Jacobian of the diffeomorphism $f$ at the
point $x$ (that is, $Jf(x) = \displaystyle\lim_{\rho \to 0} \frac{\mu( f(B_\rho(x)))}{\mu(B_\rho(x))}$),
where $B_\rho(x)$ denotes the ball of radius $\rho$ around $x.$
As it is well known,  the space is an standard Euclidean space, $Jf$ is given
by the modulus of the determinant of the derivative. In a general manifold,
the effects of changes of variables and how to compute
expressions in coordinates are well very standard.

It is immediate that $Jf$ is a multiplicative cocycle.
(that is, $Jf^{n+m} = Jf^n\circ Jf^m$). Moreover for any periodic points $p,q \in Per(f),$ with periods $k$ and $n,$ by Lemma \ref{lem3-rafael}  $$ 0 = \Lambda^s_f + \Lambda^u_f = \frac{1}{n} \log(Jf^n(p)) = \frac{1}{n}\displaystyle\sum_{j=0}^{n-1} \log(Jf(f^j(p))). $$

Since $f$ is transitive, by Livsic's theorem \cite{Livsic72}, we can solve the equation,
\begin{equation}\label{density}
\log(J f(x))  =  \phi( f(x))- \phi(x) \Leftrightarrow Jf(x) e^{-\phi(f(x))} = e^{-\phi(x)}.
\end{equation}
From \cite{LlaveMM86}, the solution $\phi$ is $C^{\infty}.$

If we consider the measure $d\nu = e^{-\phi(x)}dm$, for any borelian $S,$ we have
$$ \mu(f(S)) =  \displaystyle\int_{f(S)} e^{-\phi(x)}dm = \displaystyle\int_{S} e^{-\phi(f(y))} Jf(y) dm = \displaystyle\int_{S} e^{-\phi(y)} dm = \mu(S), $$
Since $e^{\phi(x)} \in [\frac{1}{c}, c]$ for some $c > 0,$ we can normalize $\mu$ and to obtain an $f-$invariant probability measure $\tilde{\mu},$ given by $\tilde{\mu}(S) = \frac{\mu(S)}{\mu(M)}.$

\end{proof}

\begin{lemma}\label{lemmafim}
Let $f$ be a $C^\infty$ transitive Anosov diffeomorphism of
a compact manifold $M$ satisfying  the constant periodic data condition. Then $f$ has a unique measure of maximal entropy absolutely continuous with respect to Lebesgue measure with $C^{\infty}$ density.
\end{lemma}

\begin{proof}
Since $f$ is transitive, it implies that $f$ has a unique measure of maximal entropy $\nu.$  Following the argumentation in Lemma \ref{lem3-rafael} we have

$$h_{\nu}(f) = \Lambda^u_f = \int_M \sum_{\lambda_i > 0} \lambda_i dim E_i d\nu, $$
$$ h_{\nu}(f) =  h_{\nu}(f^{-1})  = -\Lambda^s_f = -\int_M \sum_{\lambda_i < 0} \lambda_i dim E_i d\nu. $$

Using S.R.B theory, see \cite{LEDRAPPIER}, the above expressions means that $\nu$ has absolutely continuous density along the stable and unstable foliations, so $\nu$ is absolutely continuous with respect to Lebesgue measure.

Since $Jf^n(p) =1,$ for any $p$ such that $f^n(p) = p,$ by Lemma \ref{lemmafinal}, the measure $\nu$ is absolutely continuous with respect to Lebesgue measure, with $C^{\infty}$ density.
\end{proof}

\begin{remark} When $f : M \rightarrow M,$ is a $C^{1+ \alpha}, 0 < \alpha \leq 1,$ be an Anosov diffeomorphism satisfying constant periodic data condition, the solution $\phi$ of the cohomological equation as in Lemma \ref{lemmafinal} is $C^{\alpha}.$ The Theorems \ref{t1} and \ref{t2} can be formulated in $C^{1+\alpha}$ setting.
\end{remark}

\begin{theorem}\label{alfa}
Let $f: M \rightarrow M$ be a $C^{1+ \alpha}, 0 < \alpha \leq 1,$ Anosov diffeomorphism satisfying the constant periodic data condition (or every point is regular for $f$). Then $f$ has finitely many measures of maximal entropy, such that each one of them is absolutely continuous with respect to Lebesgue measure. Moreover, if $f$ has
a unique measure of maximal entropy, with continuous density with respect to Lebesgue
measure, then this measure is equivalent to Lebesgue and the density is in fact $C^{\alpha}.$
\end{theorem}

\begin{cor}\label{c1}
Let $f: M \rightarrow M$ be a $C^{1+ \alpha}, 0 < \alpha \leq 1,$  Anosov diffeomorphism satisfying constant periodic data condition  (or every point is regular for $f$).  Then $f$ is transitive if and only if $f$ has a unique measure of maximal entropy, absolutely continuous with respect to Lebesgue measure with continuous density.
\end{cor}

\begin{question}
Let $f: M \rightarrow M$ be a $C^{1+ \alpha},0< \alpha \leq 1,$  Anosov diffeomorphism satisfying constant periodic data condition  (or every point is regular for $f$). If $f$ has a unique measure of maximal entropy, is $f$ transitive ?
\end{question}

\section{Proof of Theorem \ref{t3}}

To prove Theorem \ref{t3} we will need a lemma concerning ergodic decomposition of a measure.

\begin{lemma}[Jacobs Theorem] Suppose that $M$ is a complete separable metric
space and $f: M \rightarrow M$ a continuous function. Given any borelian invariant probability measure $\mu$ let $\{\mu_P: P \in \mathcal{P}\}$ be its
ergodic decomposition. Then, $$h_{\mu}(f) = \displaystyle\int_M h_{\mu_P}(f)d\hat{\mu}(P), $$ if one side is infinite,
so is the other side.
\end{lemma}

For the proof of the above lemma we recommend \cite{OV}, page 287.

From the Jacobs Theorem, if $f$ have a unique measure $\nu$ of maximal entropy (finite), then this measure is one of the ergodic component, particularly $\nu$ is ergodic.

In the case of Theorem \ref{t3}, let $\nu$ the unique measure of maximal entropy of $f,$ we have $h_{\nu}(f) = \Lambda^u_f < \infty.$ We conclude that $\nu$ is ergodic.

Arguing as in Lemma \ref{fudamental}, since  $\nu$ is absolutely continuous with respect to Lebesgue measure with continuous density, it follows that $\nu(U) > 0,$ for any open set $U \neq \emptyset.$ It joint with the ergodicity of $\nu$ imply the transitivity of $f,$ as we did in Lemma \ref{fudamental}.

\section{Final Comments}

In \cite{Y} the author proves that there are no transitive Anosov diffeomorphisms
on negatively curved manifolds. In fact, it is a problem to know if such manifolds support Anosov diffeomorphisms. If a  negatively curved manifold $M$ supports a $C^{\infty}-$Anosov diffeomorphism $f: M \rightarrow M$ satisfying constant periodic data condition (or every point is regular for $f$), then
\begin{enumerate}
\item either $f$ has more than one measure of maximal entropy, each of these measures are absolutely continuous with respect to Lebesgue measure,
\item or if there is a unique measure of maximal entropy for $f,$ then its density with respect to Lebesgue measure is not continuous.
\end{enumerate}

%
%


\end{document}